\documentclass[12pt]{amsart}
\usepackage[latin1]{inputenc}
\usepackage{amsmath,amsthm, amscd, amssymb, amsfonts}
\usepackage[dvips]{graphicx}
\usepackage[all]{xy}
\numberwithin{equation}{section}

\theoremstyle{plain}

\numberwithin{equation}{section}

\newtheorem{theorem}[equation]{Theorem}
\newtheorem{corollary}[equation]{Corollary}

\newtheorem{lemma}[equation]{Lemma}

\theoremstyle{definition}

\newtheorem{example}[equation]{Example}

\theoremstyle{remark}
\newtheorem{remark}[equation]{Remark}

\newcommand{\Z}{{\mathbb Z}}

\newcommand\Ext{\operatorname{Ext}}
\newcommand\Jac{\operatorname{Jac}}
\newcommand\Homol{\operatorname{H}}
\newcommand\coh{\Homol}

\newcommand\op{\operatorname{op}}
\newcommand\ot{\otimes}
\newcommand\Hom{\operatorname{Hom}}

\newcommand\Hoch{\operatorname{HH}}
\newcommand\HH{\Hoch}
\newcommand\V{\mathcal{V}}

\newcommand{\DOT}{\setlength{\unitlength}{1pt}\begin{picture}(2.5,2)(1,1)\put(2,3){\circle*{2}}\end{picture}}
\newcommand{\bu}{\DOT} 
\DeclareMathOperator{\Max}{Max}

\usepackage[backend=bibtex, sorting=nyvt, maxbibnames = 9]{biblatex} 

\bibliography{References.bib}

\makeatletter
\def\blx@maxline{77}
\makeatother

\begin{document}
\title[Tensor products and support varieties]
{Tensor products and support varieties for some noncocommutative Hopf algebras}
\author{Julia Yael Plavnik}
\email{julia@math.tamu.edu}
\author{Sarah Witherspoon}
\email{sjw@math.tamu.edu}
\address{Department of Mathematics, Texas A\&M University, College Station, Texas 77843, U.S.A}
\thanks{The first author was partially supported by NSF grant DMS-1410144.
The second author was partially supported by NSF grant DMS-1401016.
}
\subjclass{16T05, 18D10}
\keywords{projective modules; nonsemisimple Hopf algebra; support varieties; smash coproduct}
\date{May 28, 2017}

\begin{abstract}
We explore questions of projectivity and tensor products of modules for
finite dimensional Hopf algebras.
We construct many classes of examples
in which tensor powers of nonprojective modules are projective
and tensor products of modules in one order are projective but
in the other order are not.
Our examples are smash coproducts with duals of group algebras,
some having algebra and coalgebra structures twisted by cocycles.
We apply support variety theory for these Hopf algebras as a tool in our investigations.
\end{abstract}

\maketitle

\section{Introduction}

Tensor products of modules for finite dimensional cocommutative Hopf algebras
(equivalently finite group schemes) are well behaved:
The tensor product is  commutative up to natural isomorphism.
Tensor powers of nonprojective modules are nonprojective.
There is a well developed theory of support varieties---a very fruitful
tool originating from finite groups---and the  variety of a
tensor product of modules is the intersection of their
varieties,  as shown by Friedlander and Pevtsova~\cite{FP}.
All of these phenomena occur  in positive characteristic. 
In characteristic~0, 
some noncocommutative Hopf algebras share this good behavior,
such as the quantum elementary abelian groups in~\cite{PW2}
by Pevtsova and the second author.

By contrast, there are examples of finite dimensional noncocommutative
Hopf algebras for which the tensor product of modules is not so
well behaved:
Benson and the second author~\cite{BW} showed that  some
Hopf algebras constructed
from finite groups  in positive characteristic 
have nonprojective modules with projective tensor powers and
modules whose tensor products in one order are projective but
in the other order are not.
A support variety theory for finite dimensional
self-injective algebras~\cite{EHTSS,SS} applies
to these examples.
The varieties are reasonably well behaved,
yet the variety of a tensor product of modules is not the intersection
of their varieties.

These results lead us to ask: What hypotheses on a noncocommutative finite dimensional
Hopf algebra are necessary or sufficient to ensure that
(1) the tensor product of modules in
one order is projective if and only if it is projective in the other order,
(2) tensor powers of  nonprojective modules are nonprojective, or
(3) the support
variety of a tensor product of modules is the intersection of their  varieties?
For a given Hopf algebra, a positive answer to
(3) implies a positive answer to (1) and (2);
this follows, e.g., from
Theorem~\ref{properties}(i) below, a restatement of a theorem of 
Erdmann, Holloway, Snashall, Solberg, and Taillefer~\cite{EHTSS}. 

In this paper we answer some of these questions for some types of Hopf algebras,
using support variety theory as our main tool. 
We generalize the examples of \cite{BW} to present many more  Hopf algebras
with reasonable support variety theory and yet 
negative answers to all three questions.
In particular we give examples in characteristic~0.
These  are in Section~\ref{three}, and are 
smash coproducts of quantum elementary abelian groups 
with duals of group algebras.
We use the variety theory developed in~\cite{PW,PW2} for 
quantum elementary abelian groups 
(that is, tensor products of Taft algebras) 
to handle these examples.
Our general Theorem~\ref{Corollary tensor product prop not preserved} shows that
these types of examples are ubiquitous in any characteristic, and are 
not just isolated anomalies:
It implies that any finite dimensional nonsemisimple Hopf algebra having a positive answer
to question (3) above can be embedded in a Hopf algebra having a negative answer to (3).

In Section~\ref{four} we give many new examples in positive characteristic by generalizing
some of the results in~\cite{BW} to crossed coproducts of group algebras
and duals of group algebras---these are smash coproducts in which the algebra and 
coalgebra structures are twisted by  cocycles.
Again we find a good support variety theory and yet 
negative answers to all three questions  above.
Group cohomology features prominently in our methods, as it did in~\cite{BW},
exploiting our Theorem~\ref{thm:iso-ad} that Hochschild cohomology
of a twisted group algebra is isomorphic to group cohomology with
coefficients in the adjoint module given by the twisted group algebra. 
Even so, 
the behavior of the representations is far removed from the usual behavior of
those of finite groups or of finite group schemes.

In contrast to the results of Sections~\ref{three} and~\ref{four},
we show in Theorem~\ref{thm:rigid}
that for  quasitriangular Hopf algebras, the
answer to question~(2) above is yes. The answer to~(1) is automatically yes 
since the tensor product of modules is commutative up to natural isomorphism.
For (2), we give an elementary argument that applies more generally to any module 
of a Hopf algebra for which the tensor product 
of the module with its dual commutes up to isomorphism.
Question (3) is open for quasitriangular Hopf algebras.

The known support variety theories needed to make sense of question (3) in general
require some homological assumptions.
We recall these assumptions 
in Section~\ref{two}, and there we also define  smash coproducts
and crossed coproducts.
In Theorem~\ref{theor_modules_bicrossed char pos} we give a general
description of tensor products and duals of modules for a crossed
coproduct of a group algebra with the dual of a group algebra. 
In Section~\ref{three}, we give general results on support variety theory
for smash coproducts with duals of group algebras,
and present our examples involving quantum elementary abelian groups.
Support varieties for 
crossed coproducts of group algebras and duals of group algebras are 
in Section~\ref{four}. 
In Section~\ref{five}, we look at consequences of commutativity,
up to isomorphism, of the tensor product of some modules.

We  work over an algebraically closed field $k$ of arbitrary characteristic,
restricting either to characteristic~0 or to positive characteristic for some
classes of examples. All modules are finite dimensional
left modules unless otherwise indicated.

\section{Smash coproducts, crossed coproducts, and support varieties}\label{two}

In this section we recall definitions and results from the literature
that we will need, and develop basic representation theory of some crossed
coproduct Hopf algebras. 

\subsection*{Smash coproducts}\label{smash-coproduct}

Smash coproduct Hopf algebras were first defined by Molnar~\cite{M}
for any commutative Hopf algebra. Here we restrict
to the commutative Hopf algebras dual to group algebras. 

Let $A$ be a Hopf algebra over $k$ and 
$G$ a finite group acting on $A$ by  Hopf algebra automorphisms.
Let $k^G$ be the Hopf algebra dual to the group algebra $kG$,
i.e.\ $k^G = \Hom_k (kG, k)$, with vector space basis 
$\{p_x : x \in G \}$ dual to $G$. 
Multiplication is given by $p_x p_y = \delta_{x,y} p_x$, 
comultiplication  $\Delta(p_x) =
\sum_{y\in G} p_y\otimes p_{y^{-1}x},$ 
counit and antipode $\varepsilon(p_x) = \delta_{1,x}$ and $ S(p_x) = p_{x^{-1}}$, for all $x,y \in G$.

The corresponding {\em smash coproduct Hopf algebra}~\cite{M} is denoted 
 $K = A \natural k^G$
and given as follows. 
The algebra structure is the usual tensor product $A\otimes k^G$ of algebras. 
Let  $a\natural p_x$ denote the element
$a\otimes p_x$ in $K$, for each $a\in A$ and $x \in G$. Comultiplication is defined by
\begin{equation}\label{eqn:coprod}
  \Delta(a\natural p_x) = \sum_{y\in G} (a_1 \natural p_y) \otimes ((y^{-1}\cdot a_2) \natural p_{y^{-1}x}),
\end{equation}
for all $x\in G$, $a\in A$. The counit and antipode are defined by
\begin{equation*}
\varepsilon(a\natural p_x) = \delta_{1,x}\varepsilon(a) \ \ 
   \mbox{ and } \ \  S(a\natural p_x) = (x^{-1}\cdot S(a)) \natural p_{x^{-1}}.
\end{equation*}
Since $A$ and $k^G$ are subalgebras of the Hopf algebra $K$, for simplicity we will sometimes write $a$ and $p_x$
in place of the elements $a\natural 1$ and $1\natural p_x$ in $K$.


Note that the elements $p_x$ are orthogonal central idempotents of $K$.
Given a $K$-module $M$ and $x\in G$, we will denote by $M_x$ the $K$-submodule $p_x\cdot M$ of $M$, which we also view as an $A$-module by restriction of action to $A$. Then 
\begin{equation}\label{eqn:G-grading}
   M = \bigoplus_{x\in G} M_x . 
\end{equation}
In this way, $K$-modules are graded by the group $G$. We call $M_x$ 
the {\em $x$-component} of $M$.
For $y \in G$, let $^{y}M_x$  denote the {\em conjugate} $K$-module: This is
$M_x$ as a vector space, with action of $A$
given by $a\cdot_y m =
(y^{-1}\cdot a) \cdot m$ for all $a\in A$, $m \in M$.
Let $M^*=\Hom_k(M,k)$, the {\em dual} of the $K$-module $M$;
the module structure is given by $(b\cdot f)(m) = f( S(b)\cdot m)$ for all
$b\in K$, $f\in M^*$, $m\in M$.

We will use the following theorem that describes what happens to the $G$-grading
when taking  tensor products and duals of $K$-modules. 

\begin{theorem}\label{theor_modules_smash}\cite[Theorem 7.1]{MVW}
Let $M$, $N$ be $A\natural k^G$-modules. For each $x\in G$, there are isomorphisms of 
$A\natural k^G$-modules: 
\begin{itemize}
\item[(i)] $(M \otimes N)_x \simeq \bigoplus_{y,z\in G, yz=x}M_y\otimes  \, ^{y}\! N_z$, and
\item[(ii)] $(M^*)_x \simeq \, ^{x}\! (M_{x^{-1}})^*$.
\end{itemize}
\end{theorem}

\subsection*{Crossed coproduct Hopf Algebras}\label{bicrossproduct}
We focus on the 
case $A=kL$, the group algebra of a finite group $L$.

Take $G$ to be a finite group acting on $L$ by automorphisms, 
and let $L$ act on $G$ trivially.
(There is a version of the following construction for nontrivial actions of $L$ on $G$;
see, e.g.,~\cite{A}.
The representation theory of the resulting crossed coproduct is more complicated,
and we do not consider it here.)

Let $(k^G)^{\times}$ denote the group of units of $k^G$.
Let $\sigma: L\times L\rightarrow (k^G)^{\times}$ be a {\em normalized 2-cocycle}, that is, a
function for which
\begin{equation}\label{cocycle}
   \sigma(l,m)\sigma(lm,n) = \sigma(m,n) \sigma(l, mn)
\end{equation}
and $\sigma(1,l)=\sigma(l,1) =1$ for all $l,m,n\in L$. We  write
$$
  \sigma(l,m) = \sum_{x\in G} \sigma_x(l,m) p_x
$$
for functions $\sigma_x : L\times L\rightarrow k^{\times}$. 
As the elements $p_x$ are orthogonal idempotents in $k^G$, the functions
$\sigma_x$ are normalized 2-cocycles with values in $k^{\times}$.
For all $x\in G$ and $l\in L$, by  
Equation~\eqref{cocycle} with $m=l^{-1}$ and $n=l$, 
$
   \sigma_x(l, l^{-1}) = \sigma_x(l^{-1},l) . $

Let $\tau: L\to k^G\otimes k^G$ be a {\em normalized 2-cocycle}, that is, 
$\tau$ is defined by 
$$
  \tau(l) = \sum_{x,y\in G} \tau_{x,y}(l) p_x\otimes p_y
$$
for functions $\tau_{x,y} : L\rightarrow k^{\times}$ satisfying 
\begin{equation}\label{eqn:tau-mult}
   \tau_{xy,z}(l)\tau_{x,y}(l) = \tau_{x,yz}(l) \tau_{y,z}(x^{-1}\cdot l)
\end{equation}
and
$
\tau_{1, x}(l)=\tau_{x,1}(l) =1
$
for all $x, y, z\in G$, $l\in L$.

Assume additionally that
\begin{equation}\label{eqn:sigma-tau-mult}
   \sigma_{xy} (l,m)\tau_{x,y}(lm) =\sigma_x(l,m)\sigma_y(x^{-1}\cdot l, x^{-1}\cdot m)\tau_{x,y}(l)\tau_{x,y}(m)
\end{equation}
for all $l,m\in L$ and $x,y\in G$.
In particular, setting $x=1$,  this implies that
   $\sigma_1(l,m)=1$, and  setting $l=1$, this implies that $\tau_{x,y}(1)=1$
for all $l,m\in L$ and $x,y\in G$.
It also follows from \eqref{eqn:sigma-tau-mult}, setting $y=x^{-1}$, that 
\begin{equation}\label{eqn:id1}
\sigma_x^{-1}(l,m) = \sigma_{x^{-1}}(x^{-1}\cdot l, x^{-1}\cdot m)\tau_{x,x^{-1}}(l)\tau_{x,x^{-1}}(m)\tau^{-1}_{x,x^{-1}}(lm) ,
\end{equation}
and setting $m=l^{-1}$, that 
\begin{equation}\label{eqn:id2}
 \tau^{-1}_{x,y}(l) = \sigma^{-1}_{xy} (l,l^{-1}) \sigma_x(l,l^{-1})\sigma_y(x^{-1}\cdot l, x^{-1}\cdot l^{-1})\tau_{x,y}(l^{-1}).
\end{equation}

We define the {\em crossed coproduct Hopf algebra} $K = kL \natural_{\sigma}^\tau k^G$
 as follows. 
As a vector space, it is
 $kL\otimes k^G$.
The product is twisted by~$\sigma$:
\begin{equation}\label{formula twisted product} (l\natural p_x) (m\natural p_y) = lm \natural \sigma(l,m) p_xp_y =
\delta_{x,y} \sigma_x(l, m) lm\natural p_x,
\end{equation}
for $l,m\in L$ and $x,y\in G$.
The coproduct is twisted by~$\tau$:
$$
   \Delta(l\natural p_x) = \sum_{y\in G}\tau_{y,y^{-1}x}(l) (l\natural p_y) \otimes ((y^{-1}\cdot l) \natural p_{y^{-1}x})
$$
for $l\in L$, $x\in G$.
The counit and antipode are given by
$$
\varepsilon(l\natural p_x) = \delta_{1,x} \ \ \ \mbox{ and } \ \ \ 
   S(l\natural p_x) = \tau_{x, x^{-1}}(l^{-1}) \sigma_x(l^{-1},l) (x^{-1}\cdot l^{-1}) \natural p_{x^{-1}}
$$
for all $l\in L$, $x\in G$.

It may be checked directly that $K$ is a Hopf algebra. 
Verification calls on assumption~(\ref{eqn:sigma-tau-mult})
as well as the cocycle defining properties~(\ref{cocycle})
and~(\ref{eqn:tau-mult}) of $\sigma$ and $\tau$.  
Our Hopf algebra $K=kL\natural_{\sigma}^{\tau}k^G$ 
is related to that given by Andruskiewitsch in~\cite[Proposition~3.1.12]{A},
but the coalgebra structure is slightly different.
We choose this version since it generalizes Molnar's smash coproduct as used in~\cite{BW}. 
See also~\cite{AD} and~\cite[Section 4]{Majid} for earlier versions
of \cite[Proposition~3.1.12]{A}.
(In comparison to~\cite{A}, we take $A=k^G$, $B=kL$, 
the weak coaction dual to our action by automorphisms of $G$ on $L$ is given by $\rho(l) = \sum_{x\in G} x^{-1}\cdot l \otimes p_x$, and the weak action $\rightharpoonup$ of $L$ on $G$ is the trivial action. Condition (3.1.9) in \cite{A} holds because the action of $G$ on $L$ is by automorphisms, and condition (3.1.11) is satisfied due to condition \eqref{eqn:sigma-tau-mult}.)

Twisted group algebras arise as subalgebras of $K$.
In general, for $\alpha: L\times L\rightarrow k^{\times}$ a 2-cocycle,
the {\em twisted group algebra} $k^{\alpha}L$ is the vector space $kL$ 
with multiplication determined by $\overline{l}\cdot \overline{m} 
=\alpha(l,m) \overline{lm}$ for all $l,m\in L$, where $\overline{l}$ is the basis
element of $k^{\alpha}L$ corresponding to $l$.
Since $\alpha$ is a 2-cocycle, we find that for all $l\in L$, 
\[
    \overline{l^{-1}} = \alpha (l^{-1},l) (\overline{l})^{-1} = \alpha (l,l^{-1})
   (\overline{l})^{-1} . 
\]

The elements $p_x$ (that is, $1\natural p_x$)
 are orthogonal central idempotents in $K$, and  
\[ K = kL\natural^{\tau}_{\sigma}k^G \simeq \bigoplus_{x\in G} kL \cdot p_x\simeq \bigoplus_{x\in G} k^{\sigma_x}L,
\]
a direct sum of ideals, where 
$k^{\sigma_x}L$ is the twisted group algebra defined above. 
The subalgebra $kL\cdot p_x$ or  $kL\ot kp_x$ of $K$ is 
indeed isomorphic to $k^{\sigma_x}L$ since 
$(l\natural p_x)(m\natural p_x) = \sigma_x(l,m) lm \natural p_x$ for all $l,m\in L$. 
Thus for a $K$-module $M$:
\[ M \simeq  \bigoplus_{x\in G} M_x,
\]
where each $M_x = p_x\cdot M$ is naturally a $k^{\sigma_x}L$-module.

In general, given a 2-cocycle $\alpha$ for $L$ with coefficients in $k^{\times}$, and an element $y$ of $G$, we can define
a new 2-cocycle $\, ^{y}\! \alpha$ for $L$ by 
\[ 
  ^{y}\! \alpha (l, m) = \alpha ( y^{-1}\cdot l, y^{-1}\cdot m) 
\] 
for all $l,m\in L$. 
We say that $\alpha$ is {\em $G$-invariant} if $ {}^y\alpha = \alpha$ for all $y\in G$. 
For each $x\in G$, the group $G$ acts on the vector space $k^{\sigma_x}L$: 
Define
$y\cdot \overline{l} = \overline{y\cdot l}$ for all $y\in G$ and $l\in L$.
A calculation shows 
that when $\sigma_x$ is $G$-invariant, this action of $G$ on $k^{\sigma_x}L$ 
is by algebra automorphisms.

We will need to use the following operations on cocycles. 
Two 2-cocycles $\alpha,\beta: L\times L\rightarrow k^{\times}$ may be multiplied:
For all $l,m\in L$, 
\[
   (\alpha\beta) (l,m) = \alpha(l,m)\beta(l,m) . 
\]
The product $\alpha\beta$ is again a 2-cocycle.
The 2-cocycle $\alpha$ has a multiplicative inverse: For all $l,m\in L$, 
\[
  \alpha^{-1}(l,m) = (\alpha(l,m))^{-1} . 
\]

Given a $K$-module $M$ and $y \in G$,  the {\em conjugate} module 
$^{y}M_x$ is $M_x$ as a vector space, with $k^{\, ^{y}\! \sigma_x}L$-module structure given by $\overline{l}\cdot_y v =
(y^{-1}\cdot \overline{l}) \cdot v$, for all $l\in L$, $v \in M$.
If $ {}^y \sigma_x\neq \sigma_x$, then $ {}^y M_x$ is not a 
$k^{\sigma_x}L$-module.

Notice that \eqref{eqn:sigma-tau-mult} may be rewritten as
\begin{equation}\label{eqn:rewrite}
   \sigma_{xy} (l,m) =\sigma_x(l,m)\sigma_y(x^{-1}\cdot l, x^{-1}\cdot m)\tau_{x,y}(l)\tau_{x,y}(m) \tau^{-1}_{x,y}(lm),
\end{equation}
or as $\sigma_{xy} = \sigma_x ({}^x\sigma_y)\beta_{xy}$ where 
$   \beta_{x,y}(l,m) = \tau_{x,y}(l)\tau_{x,y}(m) \tau^{-1}_{x,y}(lm)
$ 
for all $x,y\in G$, $l,m\in L$. 
Then $\beta_{x,y}$ is a coboundary and 
$\sigma_{xy}$ and $\sigma_x (^{x}\! \sigma_y)$ are cohomologous. 
By \cite[Lemma 1.1]{K}, there is an isomorphism of 
twisted group algebras $k^{\sigma_{xy}}L\stackrel{\sim}{\longrightarrow}k^{\sigma_x \cdot ( ^{x}\! \sigma_y)}L$ given by  $\overline{l}\rightarrow \tau_{x, y}(l) \overline{l}$.

We next give a version of Theorem~\ref{theor_modules_smash} for $K=kL\natural^{\tau}_{\sigma}k^G$.
It is stated slightly differently since the subalgebras
$k^{\sigma_x}L$ of $K$ may not be Hopf subalgebras, and $G$
may not act on $k^{\sigma_x}L$ by algebra automorphisms. 
In general, for 2-cocycles $\alpha,\beta$ on $L$ with values in $k^{\times}$,
the tensor product of a $k^{\alpha}L$-module with a $k^{\beta}L$-module has the
structure of a $k^{\alpha\beta}L$-module where $\overline{l}$ in $k^{\alpha\beta}L$ acts
as $\overline{l}\ot \overline{l}$ (the first tensor factor in 
$k^{\alpha}L$, the second in $k^{\beta}L$) by~\cite[Lemma~5.1]{K}.
For a 2-cocycle $\alpha$ on $L$ with values in $k^{\times}$,
the dual of a $k^{\alpha}L$-module $M$
has the structure of a $k^{\alpha^{-1}}L$-module 
where $\overline{l}$ in $k^{\alpha^{-1}}L$ acts as 
$(\overline{l}\cdot f)(m) = f( (\overline{l})^{-1}\cdot m)$ 
for all $m\in M$, $f\in M^*$ by~\cite[Theorem~4.6]{K}.

\begin{theorem}\label{theor_modules_bicrossed char pos}
Let $M$, $N$ be $kL\natural^{\tau}_{\sigma} k^G$-modules. 
For each $x\in G$, there are isomorphisms of $k^{\sigma_x}L$-modules: 
\begin{itemize}
\item[(i)] $(M \otimes N)_x \simeq \bigoplus_{y,z\in G, yz=x}M_y\otimes  \, ^{y}\! N_z$, and
\item[(ii)] $(M^*)_x \simeq \, ^{x}\! (M_{x^{-1}})^*$.
\end{itemize}
\end{theorem}
\begin{proof}
The proof is very similar to that of~\cite[Theorem 7.1]{MVW}.
We include details  to show the effects of $\sigma$ and $\tau$.

We will prove  statement (i) for modules of the form
$M = M_y$, $N = N_z$ for $y,z\in G$.
In principle, $M_y \otimes {}^yN_z$ is  a 
$k^{\sigma_y ( ^{y}\! \sigma_z)}L$-module. 
By equation~(\ref{eqn:rewrite}), 
$\sigma_y ({}^{y}\! \sigma_z)$ and $\sigma_{yz}$ are cohomologous, and so 
$k^{\sigma_y (^{y}\! \sigma_z)}L$ and $k^{\sigma_{yz}}L$ are isomorphic. 
This means that 
$M_y \otimes {}^y\! N_z$ can also be regarded as a $k^{\sigma_x}L$-module for $x=yz$. 
We will prove that the $k^{\sigma_x}L$-module $(M\ot N)_x$ is isomorphic to the tensor product
of the $k^{\sigma_y}L$-module $M_y$ and the $k^{\, ^{y}\! \sigma_z}L$-module $ {}^yN_z$ under this isomorphism of twisted group algebras.
The target module $M_y\ot  {}^yN_z$ is a $K$-module on which $p_{yz}$ acts as the identity and $p_w$ acts as $0$ for $w \neq yz$.
Consider the action of $l\natural p_x$, 
identified with $\overline{l}$ in $k^{\sigma_x}L$, for $l$ in $L$.
Applying $\Delta$ to $l \natural p_x$, we obtain
\[ (l\natural p_x)\cdot (m\otimes n) = \delta_{x, yz} \tau_{y, z}(l) 
  (\overline{l}\cdot m)\otimes ((y^{-1}\cdot \overline{l})\cdot n)
\]
for all $m\in M_y$ and $n\in N_z$.
On the other hand, 
the action of $\overline{l}$ in $k^{\sigma_y ({}^y\sigma_z)}L$ on the tensor product of the $k^{\sigma_y}L$-module $M_y$
and $k^{\, ^{y}\! \sigma_z}L$-module ${}^yN_z$ is given by 
\[ \overline{l}\cdot (m\otimes n) = (\overline{l}\cdot m)\otimes (\overline{ l}\cdot_{y} n)
  = (\overline{l}\cdot m)\otimes ((y^{-1}\cdot \overline{l})\cdot n).
\]
Considering the isomorphism $k^{\sigma_x}L\stackrel{\sim}{\longrightarrow}k^{\sigma_y (^{y}\! \sigma_z)}L$ given by $\overline{l}\rightarrow \tau_{y, z}(l) \overline{l}$, the $k^{\sigma_x}L$-module structure on the vector space $M_y \otimes {}^y\! N_z$ is 
precisely that arising from this isomorphism and the action of 
$k^{\sigma_y ({}^y\sigma_z)}L$ on this space. 

To prove (ii), we assume $M=M_y$ for some $y\in G$.
We first show that the dual $K$-module $M^*$ satisfies $M^* = (M^*)_{y^{-1}}$.
Let $x\in G$, $f\in M^*$, and $m\in M$.
Then
$$
  ((1\natural p_x)\cdot (f))(m) = f(S(1\natural p_x)m) =
   f((1\natural p_{x^{-1}})m) = \delta_{x^{-1},y} f(m) .
$$
So $M^*=(M_y)^*=(M^*)_{y^{-1}}$.

Setting $x=y^{-1}$, we now 
want to determine the $k^{ \sigma_{x}}L$-module structure on
$(M^*)_{x}$. 
There is an isomorphism $k^{\sigma_{x}}L \stackrel{\sim}{\longrightarrow} 
k^{{}^{x}\sigma_y^{-1}}L$, given by $\overline{l}\to \tau^{-1}_{x,y}(l)\overline{l}$,
by~(\ref{eqn:id1}).
Now $M_y$ is a $k^{\sigma_y}L$-module, so $(M_y)^*$ has the structure of a
$k^{\sigma_y^{-1}}L$-module and ${}^{x} (M_y)^*$ has the structure of
a $k^{ {}^{x}\sigma_y^{-1}}L$-module. 
Recall that if $\overline{l}\in k^{\sigma_z}L$, then $\overline{l^{-1}} = \sigma_z(l^{-1}, l)\overline{l}^{-1}$.
Then, for $m\in M_y$ and $f\in M^*_y$, by~(\ref{eqn:id2}) with $x=y^{-1}$,
\begin{eqnarray*}
   ((l\natural p_{y^{-1}})\cdot f)(m) & 
  =& f(\tau_{y^{-1},y}(l^{-1})\sigma_{y^{-1}}(l^{-1},l)((y\cdot l^{-1}) \natural p_y) m)\\
  &=& \tau_{y^{-1},y}(l^{-1})\sigma_{y^{-1}}(l^{-1},l)f\left( \overline{(y\cdot l^{-1})} m \right)\\
  &=& \tau^{-1}_{y^{-1},y} (l) f \left( \overline{(y\cdot l)}^{-1} m\right)  \\
  & = & \tau^{-1}_{y^{-1},y}(l) ((y\cdot \overline{l})\cdot f) (m), 
\end{eqnarray*}
where the overline notation is used first for elements in $k^{\sigma_y}L$,
and in the last line for elements in $k^{\sigma_y^{-1}}L$. 
The action of $y$ on $\overline{l}$ in the last line indicates that
$M_y^*$ is a $k^{ {}^x\sigma_y^{-1}}L$-module. 
Under the isomorphism $k^{\sigma_x}L\stackrel{\sim}{\longrightarrow} 
k^{ {}^x\sigma_y^{-1}}L$ given by $\overline{l}\mapsto \tau_{x,y}^{-1}(l) \overline{l}$, with $x=y^{-1}$, the $k^{ \sigma_{x}} L$-module structure on $(M^*)_{x}$ is
as claimed. 
\end{proof}

\subsection*{Support varieties}\label{supp-var}

We recall the needed definitions and results from \cite{EHTSS,SS}
on support varieties  based on Hochschild cohomology. 
There is also a parallel support variety theory using Hopf algebra cohomology,
a direct generalization of the theory for finite groups based on
group cohomology. Consequences for representation theory are the same in
either theory, although the tensor product property~(\ref{eqn:tpp}) below may not be.
We choose the Hochschild cohomology approach, as these support varieties
contain some useful extra information for our examples, namely  the $G$-components
for modules graded by the group $G$ as in~(\ref{eqn:G-grading}).

Let $A$ be a finite dimensional self-injective algebra over the field $k$, and 
let $A^e = A\otimes A^{\op}$ (the enveloping algebra of $A$).
For an $A$-bimodule $M$, 
Hochschild cohomology $\Hoch^{\bu}(A, M)$
is isomorphic to $\Ext^{\bu}_{A^e}(A, M)$;
we abbreviate $\Hoch^{\bu}(A,A)$ by $\Hoch^{\bu}(A)$.

For any left $A$-module $M$, there is a graded ring homomorphism from  
$\HH^{\bu}(A)$ to  $\Ext^{\bu}_A(M,M)$ given by $ - \ot_A M$.
Followed by Yoneda composition with generalized
extensions of $N$ by $M$, this homomorphism induces an action
of $\HH^{\bu}(A)$ on $\Ext^{\bu}(M,N)$ for any left $A$-module $N$.

In order to define support varieties, we make assumptions (fg1) and (fg2) below.
Let $\Jac(A)$ denote the Jacobson radical of $A$.

\quad

\noindent Assumption (fg1)\label{fg1}:\\
There is a graded subalgebra $H_A^{\bu}$ of $\Hoch^{\bu}(A)$ such that $H_A^{\bu}$ is a finitely
generated commutative algebra and $H_A^0 = \Hoch^{0}(A)$.

\medskip


\noindent
Assumption (fg2)\label{fg2}:\\
$\Ext^{\bu}_A(A/ \Jac(A), A/ \Jac(A))$ is a finitely generated $H_A^{\bu}$-module.

\quad

\noindent
By \cite[Proposition 2.4]{EHTSS}, 
the assumption (fg2) is equivalent to
either of: 
\begin{itemize}
\item[(i)] For all finite dimensional $A$-modules $M$ and $N$, $\Ext^{\bu}_A(M, N)$ is finitely
generated over $H_A^{\bu}$.
\item[(ii)] For all finite dimensional $A$-bimodules $M$, $\Hoch^{\bu}(A, M)$ is finitely generated
over $H_A^{\bu}$. 
\end{itemize}
Further, the above statement (i) is equivalent to the corresponding statement in which
$M = N$.


Assume (fg1) and (fg2) hold.
Let  
\[
   \mathcal V_{H^{\bu}_A} = \Max(H^{\bu}_A) , 
\] 
  the maximal ideal
spectrum of $H_A^{\bu}$.
We will also use the notation $\mathcal V _A$ when it is clear
from context which algebra $H_A^{\bu}$ is involved.

Let $M$  be a finite dimensional $A$-module. Let $I_{H^{\bu}_A}(M)$ be
the annihilator of the action of $H_A^{\bu}$
on $\Ext^{\bu}_A(M, M)$. The action may be taken to be either the left
or the right action; their annihilators coincide according to \cite[Lemma 2.1]{SS}. Then
$I_{H^{\bu}_A}(M)$ is a homogeneous ideal of $H_A^{\bu}$.
The {\em support variety} of $M$ is  
\[
   \mathcal V_{A}(M) = \Max (H^{\bu}_A / I_{H^{\bu}_A}(M)) ,
\]
the maximal ideal spectrum of  $H_A^{\bu}/I_{H^{\bu}_A}(M)$,
as in~\cite[Definition 2.3]{SS}.
We identify $\V_A(M)$ with the subset of $\V_A$ consisting of ideals 
containing $I_{H^{\bu}_A}(M)$. 

An action of a group $G$  on $A$ by algebra automorphisms induces an action
on Hochschild cohomology $\Hoch^{\bu}(A)$.
If the action of $G$ on $\HH^{\bu}(A)$ 
preserves $H_A^{\bu}$, then $G$ also acts on subvarieties of $\mathcal V_A$
in such a way that $\mathcal V _{A} ( \, {}^g \! M) = \, {}^g \mathcal V _{A}(M)$
for all finite dimensional $A$-modules $M$ and group elements $g\in G$.

We will use  the following result from \cite{EHTSS,SS}.
Part~(i) follows from~\cite[Theorem~2.5(c)]{EHTSS}, 
which states more generally that the dimension of a support variety
is the complexity of the module (that is, the rate of growth
of a minimal projective resolution). 
Part~(ii) is~\cite[Proposition~3.4(f)]{SS}.

\begin{theorem}\cite{EHTSS,SS}\label{properties}
Let $A$ be a finite dimensional self-injective algebra  satisfying (fg1) and (fg2),
and let $M$ and $N$ be finite dimensional $A$-modules. Then
\begin{itemize}
\item[(i)] $\dim(\V_A(M))= 0$ if and only if $M$ is projective.
\item[(ii)] $\V_A(M\oplus N)=\V_A(M)\cup\V_A(N)$.
\end{itemize}
\end{theorem}

Now let $A$ be a finite dimensional Hopf algebra that  satisfies (fg1) and (fg2).
We say that $A$ has the {\em tensor product property} with respect to $H^{\bu}_A$ if
\begin{equation}\label{eqn:tpp}
   {\mathcal V}_A(M\ot N) = {\mathcal V}_A(M)\cap  {\mathcal V}_A(N)
\end{equation}
for all finite dimensional $A$-modules $M$ and $N$.
For example, the tensor product property holds 
if $A$ is the group algebra of a finite $p$-group $G$ over
a field $k$ of characteristic $p$
and $H^{\bu}_A$ is the subalgebra of $\HH^{\bu}(A)$ generated by 
group cohomology $\Ext^{\bu}_{kG}(k,k)$ and $\HH^0(kG)$. 
(For this, we use a standard embedding of group cohomology into
Hochschild cohomology of the group algebra; see, e.g.,  Corollary~\ref{cor:embedding} below,
taking the cocycle $\alpha$ to be trivial.)
For finite groups in general, the tensor product property holds for the 
classical choice of support variety theory based on group cohomology instead.
(If $G$ is not a $p$-group, $\HH^0(kG)$ is no longer a local ring, and its primitive
central idempotents contribute to the support varieties in the Hochschild
cohomology version.)   
Cocommutative Hopf algebras generally~\cite[Proposition~3.2]{FP}
have the tensor product property for support varieties defined in terms of their
Hopf algebra cohomology. 
Quantum elementary abelian groups~\cite[Theorem~1.1]{PW2}
have the tensor product property for either version of support varieties
(these algebras have trivial centers).

\section{Properties of modules for smash coproducts}\label{three}
Let $K = A\natural k^G$ be a smash coproduct  arising
from a finite dimensional Hopf algebra $A$ with an action of a finite group $G$ by Hopf algebra automorphisms, as described in Section~\ref{two}.
In this section, we study homological and representation theoretic 
properties of  $K$.

\subsection*{Smash coproducts and finite generation assumptions}
We will assume that $A$ satisfies the conditions (fg1) and (fg2) of Section~\ref{supp-var}.
The following theorem shows that $K$ satisfies these two conditions as well, with a suitable choice of $H^{\bu}_K$. 

\begin{theorem}\label{thm:fg12}
Let $A$ be a finite dimensional Hopf algebra with an action
of a finite group $G$ by Hopf algebra automorphisms.
Assume that $A$ satisfies conditions (fg1) and (fg2) from Section~\ref{supp-var}.
Then the smash coproduct $K = A\natural k^G$ also satisfies conditions (fg1) and (fg2).
\end{theorem}

\begin{proof}
Given that $A$ satisfies condition (fg1), there exists a graded subalgebra $H_A^{\bu}$ of $\Hoch^{\bu}(A)$ such that $H_A^{\bu}$ is a finitely generated commutative algebra and $H_A^0 = \Hoch^{0}(A)$.
As an algebra, $K$ is  the tensor product algebra $A\ot k^G$, and so
\[\Hoch^{\bu}(K) = \Hoch^{\bu}(A\ot k^G) \simeq
   \Hoch^{\bu}(A)\ot \Hoch^{\bu}(k^G)\simeq  \Hoch^{\bu}(A)\ot k^G.
\]
The last isomorphism follows from the semisimplicity and commutativity of $k^G$;
the tensor factor $k^G$ is concentrated in homological degree~0.
Consider the subalgebra
$$H_K^{\bu} = H_A^{\bu}\otimes k^G$$
of $\Hoch^{\bu}(A)\otimes k^G \simeq \Hoch^{\bu}(K)$.
By its definition, $H_K^{\bu}$ is a graded subalgebra of $\Hoch^{\bu}(K)$ and
$H_K^0 = H_A^0\otimes k^G$ coincides with $\Hoch^{0}(A)\otimes k^G  = \Hoch^{0}(K)$.
In addition, $H_K^{\bu}$ is a finitely generated commutative algebra since both $H_A^{\bu}$ and $k^G$ are.
In this way, $K$ satisfies (fg1) via $H_K^{\bu}$.

To check that $K$ also satisfies (fg2),
let $M$ and $N$ be finite dimensional $K$-modules.
Recall that $M = \oplus_{x\in G}M_x$ and $N=\oplus_{x\in G}N_x$ where $M_x = p_x\cdot M$
and $N_x = p_x\cdot N$ for each $x\in G$.
Since $\{p_x: x\in G\}$ is a set of orthogonal central idempotents in $K$ and
$K p_x \simeq A$ for each $x$,
it also follows that
\[
  \Ext^{\bu}_K(M,N)\simeq \bigoplus_{x\in G} \Ext^{\bu}_K (M_x,N_x) 
       \simeq \bigoplus_{x\in G} \Ext^{\bu}_A(M_x,N_x),
\]
and the action of $\HH^{\bu}(K)$ on $\Ext^{\bu}_K(M_x,N_x)$ corresponds
to the action of $\HH^{\bu}(A)$.
By hypothesis, for each $x\in G$, $\Ext^{\bu}_A(M_x,N_x)$ is finitely generated as
an $H_A^{\bu}$-module, and since $H_A^{\bu}\ot kp_x\subset H_K^{\bu}$,
$\Ext^{\bu}_K(M,N)$ is finitely generated over $H_K^{\bu}$.
\end{proof}

Since $K = A\natural k^G$ satisfies (fg1) and (fg2) once
$A$ does, we may define support varieties for $K$-modules, as in Section~\ref{supp-var},
for some classes of Hopf algebras $A$.

\subsection*{Smash coproducts and the tensor product property}
The following theorem is a consequence of Theorem~\ref{properties}(ii).
As in the proof of Theorem~\ref{thm:fg12}, we choose $H_K^{\bu} = H_A^{\bu}\ot k^G$.
Then  $\mathcal V_K$, the maximal ideal spectrum of $H_K^{\bu}$, 
may be identified with $\mathcal V_A \times G$, whose elements we denote by
$I\times x$ for $I$ a maximal ideal of $H^{\bu}_A$ and $x\in G$.
For simplicity of language we do not  mention the choices $H^{\bu}_A$,
$H^{\bu}_K$ explicitly in theorem statements. 

\begin{theorem}\label{thm:formulas}
Let $A$ be a finite dimensional Hopf algebra with an action of a finite group $G$ by
Hopf algebra automorphisms. Assume that $A$ satisfies (fg1) and (fg2) and 
has the tensor product property~(\ref{eqn:tpp}).
Let $M$ and $N$ be  finite dimensional $A\natural k^G$-modules. Then
\begin{itemize}
\item [(i)] $\mathcal V_K(M) = \bigcup_{x\in G} \mathcal V_A (M_x)\times x$.
\item[(ii)] $\mathcal V_A((M\otimes N)_x) = \bigcup_{y,z\in G, \ yz = x} \mathcal V_A(M_y)\bigcap \,
       ^{y}(\mathcal V_A(N_z))$.
\end{itemize}
\end{theorem}
\begin{proof}
To prove statement (i), note that the variety of a direct sum is the union of the varieties, by Theorem~\ref{properties}(ii).
Applying this result to $M = \oplus_{x\in G} M_x$, we have $\mathcal V_K(M) = \bigcup_{x\in G} \mathcal V_K (M_x) = \bigcup_{x\in G} \mathcal V_A (M_x)\times x$.

Since $A$ has the tensor product property, statement (ii) is a consequence of Theorem \ref{theor_modules_smash}(i) and Theorem~\ref{properties}(ii):
\begin{eqnarray*}
\mathcal V_A((M\ot N)_x)  & = & \bigcup _{y,z\in G, \ yz=x} \mathcal V_A(M_y\ot \, {}^y \! N_z) \\
   & = & \bigoplus_{y,z\in G, \ yz=x} \mathcal V_A(M_y)\cap \mathcal V_A( \, {}^y \! N_z) \\
   &=& \bigoplus _{y,z\in G, \ yz=x} \mathcal V_A(M_y)\cap \, ({}^y  \mathcal V_A(N_z)).
\end{eqnarray*}
\end{proof}

We next show that for any finite dimensional nonsemisimple Hopf algebra $A$
having the tensor product property, there is a Hopf algebra containing $A$ as a
 subalgebra that does not.
There are many such Hopf algebras, but we will take 
a smash coproduct $K = (A\ot A)\natural k^{\Z_2}$ where $\Z_2$ is
the group of order two in which the nonidentity element interchanges
the tensor factors. 


\begin{theorem}\label{Corollary tensor product prop not preserved}
Let $A$ be a finite dimensional nonsemisimple Hopf algebra satisfying
(fg1) and (fg2) and the tensor product property~(\ref{eqn:tpp}).
Let $K=(A\ot A)\natural k^{{\mathbb Z}_2}$.
Then $K$ does not have the tensor product property.
Moreover, there are
nonprojective $K$-modules $M$ and $N$ such that $M\otimes M$ and $M\otimes N$ are
projective, and $N\otimes M$ is not projective.
\end{theorem}

The statement and proof of the theorem is based on choices $H^{\bu}_A$
and $H^{\bu}_K$ discussed above. However the last statement in the theorem
implies that there is no choice of $H^{\bu}_K$ with respect to which $K$ has the 
tensor product property.
Even further, it implies that replacing these support varieties with those defined
via Hopf algebra cohomology $\Ext^{\bu}_K(k,k)$,
the tensor product property still will not hold for $K$,
due to Theorem~\ref{properties}(i). 

\begin{proof}
Note that $\Hoch^{\bu}(A\ot A) \simeq \Hoch^{\bu}(A)\otimes \Hoch^{\bu}(A)$. We choose
 $H^{\bu}_{A\ot A}=H_A^{\bu}\ot H_A^{\bu}$, which  satisfies (fg1) and (fg2)
for the Hopf algebra $A\otimes A$:
Condition (fg1) holds since $H^{\bu}_{A\ot A}$ is the tensor product of two copies
of $H^{\bu}_A$. For (fg2), note that
$$
\begin{aligned}
   &\Ext^{\bu}_{A\ot A}((A\ot A)/\Jac(A\ot A), (A\ot A)/\Jac(A\ot A))
\\ & \hspace{1cm} \simeq \Ext^{\bu}_A (A/\Jac(A), A/\Jac(A))\ot
\Ext^{\bu}_A(A/\Jac(A), A/\Jac(A)).
\end{aligned}
$$

Since $A$ is not semisimple, there exists a nonprojective $A$-module $U$,
and so by Theorem~\ref{properties}(i),
$\dim({\mathcal V}_A(U))\neq 0$.
We will be interested in the two nonprojective $A\ot A$-modules given by $U\ot A$ and $A\ot U$,
viewing $A$ as the left regular $A$-module.
Note that $\Ext^{\bu}_{A\ot A}(U\ot A, U\ot A)\simeq 
\Ext^{\bu}_A(U,U)$, since $A$ is projective as an $A$-module, where the
subscript $A$ is understood to be identified with $A\ot k$ as a subalgebra of $A\ot A$.
Similarly $\Ext^{\bu}_{A\ot A}(A\ot U, A\ot U)\simeq \Ext^{\bu}_A(U,U)$, only
this time the subscript $A$ is identified with $k\ot A$. Then   
$\dim({\mathcal V}_{A\ot A}(U\ot A)\cap {\mathcal V}_{A\ot A}(A\ot U)) = 0$.

Let $h$ be the nonidentity element of $\mathbb Z _2$,
and let $M = (U\ot A)\ot kp_h$ as a $K$-module.
Then $M$ is nonprojective since $U\ot A$ is a nonprojective $A$-module. 
We have $\dim({\mathcal V}_K(M)) = 
\dim({\mathcal V}_{A\ot A}( U \ot A) \times h) \neq 0$.
Note that ${}^h (U\ot A) \simeq A\ot U$ since
$h$ interchanges the two factors of $A$ in the tensor product $A\ot A$.
Thus by Theorem~\ref{theor_modules_smash}(i),
\[
   M\ot M \simeq ((U\ot A) \ot (A\ot U))\ot kp_1.
\]
It follows from Theorem~\ref{thm:formulas}(i) and (ii) that
 $\dim({\mathcal V}_K(M\ot M)) = 0$.
By Theorem~\ref{properties}(i), $M\ot M$ is projective. 
On the other hand, $\dim(\V_K(M)\cap \V_K(M))\neq 0$ since $M$ is not projective.
Therefore $K$ does not have the tensor product property.
Consider the $K$-module $N = (U\otimes A)\otimes kp_1$, also nonprojective. 
By Theorem~\ref{theor_modules_smash}(i),
\[
   M\ot N \simeq ((U\ot A) \ot (A\ot U))\ot kp_h,
\]
and 
\[
   N\ot M \simeq ((U\ot A) \ot (U\ot A))\ot kp_h.
\]
Then, $\dim({\mathcal V}_K(M\ot N)) = 0$ and 
$\dim({\mathcal V}_K(N\ot M)) = \dim ({\mathcal V}_A(U\ot A)) \neq 0$
by Theorem~\ref{thm:formulas}(i) and (ii) and since $U\ot A$
is a nonprojective $A$-module. 
This finishes the proof, since projectivity is determined by the dimension of the variety by Theorem~\ref{properties}(i).
\end{proof}

Next we show that $A$ embeds into $K$ as a subalgebra.
Examination of the coproduct~(\ref{eqn:coprod}) shows that
$A$ is not a Hopf subalgebra.

\begin{corollary}
Let $A$ be a finite dimensional nonsemisimple Hopf algebra satisfying
(fg1) and (fg2) and the tensor product property.
Then $A$ can be embedded as a subalgebra of a Hopf algebra that does not have
the tensor product property.
\end{corollary}

\begin{proof}
Let $K = (A\ot A)\natural k^{\Z_2}$ as in the theorem.
Then $(A\ot k)\natural k$ and $(k\ot A)\natural k$ are
subalgebras of $K$ that are isomorphic to $A$.
\end{proof}

Next we will give some specific examples of  varieties and of 
further constructions involving projective and nonprojective modules.

\subsection*{Smash coproducts with quantum elementary abelian groups}

Let $k$ have characteristic 0.
Let $n\geq 2$ be a positive integer, and let $q$ be a primitive $n$th root of 1 in $k$.
The {\em Taft algebra} $T_n$ is the algebra generated by symbols
$g$ and $x$ with relations $x^n = 0$, $g^n = 1$ and $gx = q xg$.
It is a Hopf algebra with comultiplication given by $\Delta(g) = g\otimes g$ and $\Delta(x) = 1\otimes x + x\otimes g$,  counit $\varepsilon(g) = 1$ and $\varepsilon(x)= 0$, and
antipode $S(g)=g^{-1}$ and $S(x)=- x g^{-1}$.

A {\em quantum elementary abelian group} is a tensor product $A$ of $m$ copies of
the Taft algebra $T_n$, for some positive integer $m$.
It is isomorphic to a skew group algebra $\Lambda\rtimes G$, where the group $G \simeq(\mathbb Z_n)^m$ is elementary abelian with generators $g_1,\ldots, g_m$,
and  $\Lambda = k[x_1, x_2, \ldots, x_m]/ (x_1^n, x_2^n,\ldots, x_m^n)$, see~\cite{PW}. In this way, we view $A$ as a Radford biproduct (or bosonization of a Nichols algebra), see~\cite[Exercise 8.25.12]{EGNO}.

The cohomology of the quantum elementary abelian group $A$ is
$$\Homol^{\bu}(A, k) = \Ext^{\bu}_A(k, k)\simeq \Ext^{\bu}_{\Lambda}(k,k)^G
\simeq k[y_1,\ldots, y_m]$$
as a graded algebra, where $\deg(y_i) = 2$. This was shown explicitly for the case 
$m=1$ for example in \cite[Section 4]{PW} and the general case follows by applying the K\"{u}nneth formula.
Let $H_A^{\bu} = \Homol^{\bu}(A,k)$.
Condition (fg1) is satisfied; note that the center of a Taft
algebra is just all scalar multiples of 1.
Condition (fg2) holds as well: This is a direct consequence of
the isomorphism $$\Ext_A^{\bu}(M, N)\simeq \Ext_A^{\bu}(k, N\otimes M^{*})$$
and \cite[Lemma 7]{PW}.
It was shown in \cite[Theorem 1.1]{PW2} that $A$ has the tensor product property.

In \cite{PW} it is shown that the support
variety of a finite dimensional $A$-module is homeomorphic to its rank variety.
The rank variety of a finite dimensional $A$-module $U$ is defined as follows.
For each $m$-tuple of scalars $(\lambda_1,\ldots,\lambda_m)\in k^m$, let
$$
  \tau(\lambda_1,\ldots,\lambda_m) = \lambda_1 x_1+ \lambda_2 x_2 g_1 +\cdots
   + \lambda_m x_m g_1 g_2\cdots g_{m-1} .
$$
It can be shown that $\tau(\lambda_1,\ldots,\lambda_m)^n =0$ in $A$, and that $\tau(\lambda_1,\ldots,\lambda_n)$
generates a subalgebra $k\langle \tau(\lambda_1,\ldots,\lambda_m)\rangle$
of $A$ isomorphic to $k[t]/(t^n)$.
Define the {\em rank variety} of an $A$-module $U$ to be
\begin{equation}\label{eqn:rank}
   \V_A ^ r (U) = \{ (\lambda_1,\ldots, \lambda_m)\in k^m : 
     U\downarrow_{k \langle \tau(\lambda_1,\ldots,\lambda_m)\rangle }
  \mbox{ is  not projective}\}/ G,
\end{equation}
where the quotient indicates the orbit space under the action of $G=\langle g_1,\ldots, g_m\rangle$
on  $k^m$ by $g_i \cdot e_j = q^{\delta_{ij}} e_j$ (the $e_j$ are the standard
basis vectors).
The downarrow notation denotes restriction of a module to a subalgebra.
Compare to~\cite[Definition~3.2]{PW}, in which the rank variety is defined
instead as the corresponding projective variety.

Next we give some examples where the rank variety
can be determined fairly quickly.

\begin{example}\label{ex:Sweedler}
We will look more closely at the specific example
$A=H_4\otimes H_4$, the tensor product of two copies of the Sweedler Hopf algebra
$H_4 = T_2$.
Let $G = \mathbb Z_2 = \langle h \rangle$, acting on $A$
by interchanging the generators of the first and second copies of $H_4$ in $A$, that is, $h\cdot g_1 = g_2$, $h\cdot g_2=g_1$,  $h\cdot x_1 = x_2$,
and $h\cdot x_2 = x_1$.
Let $K = A\natural k^G$.

Consider the $A$-module $U = \overline{A}$ given by the quotient of the left regular module $A$ by the ideal generated by $x_2$ and $g_2 - 1$. Then $U\simeq H_4\ot k$, the first copy of the Sweedler algebra in $A$.

We will find the rank variety~(\ref{eqn:rank}) of the $A$-module $U$, which 
is homeomorphic to the support variety. 
The restricted module $U\downarrow_{k\langle x_1\rangle}$ is isomorphic to the direct sum of two copies of the right regular module
$k \langle x_1\rangle$, indexed by 1 and $g_1$, since this is the action of $x_1$
on the regular module $H_4$.
Therefore  $U\downarrow_{k\langle x_1\rangle}$
is projective as a $k\langle x_1\rangle$-module.
In this case, we are letting $(\lambda_1,\lambda_2) = (1,0)$ in the notation
described above.
Similarly, whenever $\lambda_1\neq 0$, the restricted module 
$U\downarrow_{k\langle \lambda_1x_1+\lambda_2x_2\rangle }$ is projective 
since it is isomorphic to the left regular
$k\langle \lambda_1x_1+\lambda_2x_2\rangle$-module generated by $1\ot 1$. 

Now, we consider $(\lambda_1,\lambda_2) = (0,1)$. The restricted module $U\downarrow_{k\langle x_2 g_1\rangle}$ is
a trivial module.
So $U\downarrow_{k\langle x_2 g_1\rangle}$ is not projective.
The rank variety of $U$ is thus 
\[
  V_A^r (U) = \{ (0,\lambda_2) :  \lambda_2\in k \} /G.
\]
The conjugate module ${}^h U$, on the other hand, is trivial on
restriction to $k\langle x_1 \rangle$ and free on restriction to
$k\langle x_2 g_1 \rangle$, and so  
\[
  \V_A^r( {}^h U) = \{ (\lambda_1,0) : \lambda_1\in k \} /G.
\]

We will next use the $A$-module $U$ and the trivial $A$-module $k$ 
to construct some specific $K$-modules  and study their properties. 

\begin{itemize}
\item Consider the $K$-modules $M = U\otimes kp_1$ and $N = k\otimes kp_h$. We want to compare $M\otimes N$ with $N\otimes M$.
First notice that $\mathcal V_K(M)\bigcap\mathcal V_K(N) =0$, since $M$ is concentrated on the $1$-component while $N$ is concentrated on the $h$-component.

On the other hand, by Theorem~\ref{thm:formulas},
$\mathcal V_K(M\otimes N) = \mathcal V_A(U)\times h \neq 0$ and $\mathcal V_K(N\otimes M) = \, ^{h}(\mathcal V_A(U))\times h\neq 0$.
So neither $\mathcal V_K(M\otimes N)$ nor $\mathcal V_K(N\otimes M)$ is contained in $\mathcal V_K(M)\bigcap\mathcal V_K(N)= 0$.
As a consequence, $K$ does not inherit the tensor product property from $A$.
(See the proof of 
Theorem~\ref{Corollary tensor product prop not preserved} for a different
example illustrating this phenomenon.)

Moreover, since $\mathcal V_A(U) \neq \, ^{h}(\mathcal V_A(U))$, we see that
$$\mathcal V_K(M\otimes N)\neq \mathcal V_K(N\otimes M).$$
 Thus the order of the tensor product matters in computing the varieties.

\item Consider the $K$-module $M = U\otimes kp_h$. 
We claim $\mathcal V_K(M)\neq\mathcal V_K(M^*)$: 
In \cite{PW2}, to prove the tensor product property for $A$, the authors use the 
fact that $\mathcal V_A(U) = \mathcal V_A(U^*)$, for every finite dimensional $A$-module $U$.
But this is not true for $K$. For example,
since $h^2 = 1$, it follows from Theorem~\ref{theor_modules_smash}(ii) and this
fact about varieties of $A$-modules
that $\mathcal V_K(M^*) = \mathcal V_A(^{h}U^*)\times h  = {}^{h} \mathcal V_A(U^*)\times h = {}^{h} \mathcal V_A(U)\times h $
while $\mathcal V_K(M) = \mathcal V_A(U) \times h$, and these two varieties are different.

\item Consider the $K$-modules $M = U\otimes kp_h$ and $N = U\otimes kp_1$. We will show that $M\otimes N$ is projective while $N\otimes M$ is not.

Since $\mathcal V_K(M\otimes N) = \mathcal V_A(U\otimes \, ^{h} U)\times h$ and $A$ has the tensor product property, 
it follows that $\mathcal V_K(M\otimes N) =0$. Then $M\otimes N$ is a projective $K$-module by Theorem~\ref{properties}(i).

On the other hand, $N\otimes M$ is not projective since
$$\mathcal V_K(N\otimes M) = \mathcal V_A(U\otimes U)\times h = \mathcal V_A(U)\times h , $$
and the dimension of this variety is not~0. 

\item Consider the $K$-module $M = U\otimes kp_h$. We will show that $M\otimes M$ is projective while $M$ is not:
 $\mathcal V_K(M\otimes M) = \mathcal V_A(U\otimes ^{h}U)\times 1 =0$ since $A$ has the tensor product property. Then $M\otimes M$ is a projective $K$-module
as a consequence of Theorem~\ref{properties}(i).
\end{itemize}
\end{example}

\begin{example}
Let $A= H_4^{\ot m}$, that is, $m$ copies of the Sweedler Hopf algebra $H_4$.
Let $G={\mathbb{Z}}_m=\langle h\rangle$, with $h$ acting on $A$ by cyclically permuting the
tensor factors. Let $U=H_4\ot k\ot\cdots\ot k$, so that the first copy of $H_4$ in $A$
acts on $U$ as on the left regular module, while each of the others acts as on a trivial module.
Let $M = U\ot kp_h$. Then by reasoning similar to that in Example~\ref{ex:Sweedler}, 
$$
   \V_k(M^{\ot m})= \V_A(U\ot {}^h U\ot\cdots\ot {}^{h^{m-1}} U)\times h^m
    = \left(\bigcap_{i=0}^{m-1} {}^{h^{i}}( \V_A(U))\right)\times h^m=0,
$$
while $\dim(\V_K(M^{\ot (m-1)}))\neq 0$. 
Therefore $M^{\ot m}$ is projective and $M^{\ot (m-1)}$ is not projective.
\end{example}

\section{Properties of modules for crossed coproducts}\label{four}
Let $k$ have positive characteristic $p$.
In this section we consider modules for $K=kL\natural _{\sigma}^{\tau} k^G$, 
a crossed coproduct algebra as defined in Section~\ref{two}, when
the group $L$ has order divisible by $p$. 
We showed in Section~\ref{two} 
that as an algebra, $K\simeq \oplus_{x\in G} k^{\sigma_x} L$,
a direct sum of the twisted group algebras $k^{\sigma_x}L$.
We first find a description of their cohomology. 

\subsection*{Hochschild cohomology of twisted group algebras}
Let $\alpha$ be a normalized 2-cocycle on the group 
$L$ with values in $k^{\times}$.

First note that $L$ acts on the twisted group algebra $k^{\alpha}L$ by automorphisms:
$$
   l ( \overline{m}) = \overline{l}\cdot \overline{m}\cdot (\overline{l})^{-1}
     = \alpha^{-1}(l^{-1},l)\overline{l}\cdot \overline{m}\cdot \overline{l^{-1 }}
$$
for all $l,m\in L$.
We call this the {\em adjoint action} and denote by $(k^{\alpha}L)^{ad}$
the corresponding $kL$-module.
One may check, using the cocycle condition, that this is indeed an action.
A similar calculation shows that there is an injective algebra homomorphism from
the group algebra $kL$ to the enveloping algebra
$( k^{\alpha}L)^e$ of the twisted group algebra $k^{\alpha}L$, given by
\begin{eqnarray*}
\delta: kL & \longrightarrow & k^{\alpha}L \ot (k^{\alpha}L)^{op} , \\
      l & \mapsto & \overline{l}\ot (\overline{l})^{-1} ,
\end{eqnarray*}
for all $l\in L$.
Identify $kL$ with its image $\delta(k L)$ in $(k^{\alpha}L)^e$.
In the following statement, the uparrow denotes tensor induction, that is,
$M\!\uparrow_A^B = B\ot_A M$ for all $A$-modules $M$, where $A$ is
a subalgebra of $B$.
The action of $B$ on $M\!\uparrow_A^B$ is by multiplication on the left tensor factor.

\begin{lemma}\label{lem:induced}
As a $(k^{\alpha}L)^e$-module, $k^{\alpha}L \simeq k\uparrow_{kL}^{(k^{\alpha}L)^e} $.
\end{lemma}

\begin{proof}
Let $f: k^{\alpha}L\rightarrow k\uparrow_{kL}^{(k^{\alpha}L)^e} =
(k^{\alpha}L)^e\ot_{\delta (kL)} k$ be the $k$-linear function defined by
$f(\overline{l})= (\overline{l}\ot \overline{1})\ot_{\delta (kL)} 1$ for all $l\in L$.
Clearly $f$ is bijective.
We check that it is a $(k^{\alpha}L)^e$-module homomorphism:
Let $l, m, n \in L$. Then
\begin{eqnarray*}
 f((\overline{m}\ot\overline{n}) \cdot \overline{l}) & = & f (\overline{m}\cdot
                \overline{l}\cdot \overline{n})\\
   & = & \alpha(m,l) \alpha(ml,n) f(\overline{mln}) \\
   &=& \alpha(m,l) \alpha(ml,n) (\overline{mln}\ot \overline{1})\ot_{\delta(kL)} 1 ,
\end{eqnarray*}
and on the other hand, since the tensor product is taken over $\delta(kL)$,
\begin{eqnarray*}
    (\overline{m}\ot \overline{n}) \cdot f(\overline{l}) & = & (\overline{m}\ot \overline{n})
   \cdot (\overline{l}\ot \overline{1})\ot _{\delta(kL)} 1 \\
    &=& \alpha(m,l) (\overline{ml}\ot \overline{n})\ot _{\delta(kL)} 1\\
    &=& \alpha(m,l) (\overline{ml}\ot \overline{n}) (\overline{n}\ot
    (\overline{n})^{-1}) \ot_{\delta(kL)} 1 \\
   &=& \alpha(m,l) \alpha(ml,n) (\overline{mln}\ot \overline{1})\ot_{\delta(kL)} 1 .
\end{eqnarray*}\end{proof}

As a consequence of Lemma~\ref{lem:induced}, we may apply the
Eckmann-Shapiro Lemma~\cite[Corollary~2.8.4]{Benson} to obtain
a result on Hochschild cohomology next.
For a $kL$-module $M$, the {\em group cohomology} of $L$
with coefficients in $M$ is 
\[
   \coh^{\bu}(L,M)=\Ext^{\bu}_{kL}(k,M) . 
\]
If $M$ is in addition an algebra and the action of $L$ is by 
algebra automorphisms,
then $\coh^{\bu}(L,M)$ is an algebra under cup product
(corresponding to tensor product of generalized extensions) followed by the
multiplication map $M\ot M\rightarrow M$.
In this way, $\coh^{\bu}(L, (k^{\alpha}L)^{ad})$ is itself an algebra.
We show that this is none other than the Hochschild cohomology 
of the twisted group algebra $k^{\alpha}L$, generalizing the
well known result for the group algebra $kL$ 
(see, e.g.,~\cite[Proposition~3.1]{SW}). 

\begin{theorem}\label{thm:iso-ad}
Let $L$ be a finite group and let $\alpha: L\times L\rightarrow k^{\times}$
be a 2-cocycle.  
There is an isomorphism of algebras,
$\HH^{\bu}(k^{\alpha}L) \simeq \coh^{\bu}(L, (k^{\alpha}L)^{ad})$.
\end{theorem}

\begin{proof}
The enveloping algebra $(k^{\alpha}L)^e$ of $k^{\alpha}L$ 
is free as a right $kL$-module where $kL$ acts via the embedding $\delta$:
Elements of the form $1\ot \overline{l}$ ($l\in L$)
constitute a set of free generators. 

Since $(k^{\alpha}L)^e$ is flat as a $kL$-module, we may apply the 
Eckmann-Shapiro Lemma
in combination with Lemma~\ref{lem:induced} to obtain an
isomorphism of graded vector spaces,
\begin{eqnarray*}
   \HH^{\bu} (k^{\alpha}L) & = & \Ext^{\bu}_{(k^{\alpha}L)^e} (k^{\alpha}L, k^{\alpha}L)\\
    & \simeq & \Ext^{\bu}_{(k^{\alpha}L)^e} (k\uparrow_{kL}^{(k^{\alpha}L)^e} , k^{\alpha}L)\\
   & \simeq & \Ext^{\bu} _{kL} (k, k^{\alpha}L). 
\end{eqnarray*}
Note that the corresponding action of $kL$ on $k^{\alpha}L$ is precisely
the adjoint action, since $kL$ embeds into $(k^{\alpha}L)^e$ as $\delta(kL)$.
Finally, one may check that this is in fact an isomorphism of
algebras, using essentially the same proof 
as~\cite[Proposition~3.1]{SW}.
\end{proof}

\begin{corollary}\label{cor:embedding}
Let $L$ be a finite group and let $\alpha:L\times L\rightarrow k^{\times}$ be a 2-cocycle.
The group cohomology ring $\coh^{\bu}(L,k)$ embeds as a subalgebra
of the Hochschild cohomology ring $\HH^{\bu}(k^{\alpha}L)$ of the twisted group algebra $k^{\alpha}L$.
\end{corollary} 

\begin{proof}
By Theorem~\ref{thm:iso-ad}, $\HH^{\bu}(k^{\alpha}L)\simeq
\coh^{\bu}(L, (k^{\alpha}L)^{ad})$.  
As a $kL$-module, $(k^{\alpha}L)^{ad}$ decomposes into a direct sum
of submodules indexed by conjugacy classes of $L$.
In particular, the one-dimensional 
direct summand $k\cdot \overline{1}$ of $(k^{\alpha}L)^{ad}$
is a trivial $kL$-module.
The corresponding direct summand $\Ext^{\bu}_{kL}(k, k\cdot \overline{1})$ of
$\Ext^{\bu}_{kL}(k, (k^{\alpha}L)^{ad})$ is  a subalgebra under cup product, since $k\cdot \overline{1}$ is a subalgebra of $k^{\alpha}L$.
This summand is precisely
the group cohomology $\coh^{\bu}(L,k) = \Ext^{\bu}_{kL}(k,k)$.
\end{proof}

\subsection*{Crossed coproducts and finite generation assumptions}
We now show that the crossed coproducts $kL\natural_{\sigma}^{\tau} k^G$ 
satisfy our assumptions (fg1) and (fg2). 

\begin{theorem}
Let $K=kL\natural^{\tau}_{\sigma} k^G$ be a crossed coproduct Hopf algebra
as defined in Section~\ref{two}. 
Then $K$ satisfies conditions (fg1) and (fg2) of Section~\ref{supp-var}.
\end{theorem}

\begin{proof}
As conditions (fg1) and (fg2)  depend only on the algebra structure of $K$,
the theorem statement does not depend on $\tau$. 
We view $K$ as a direct sum of the twisted group algebras
$k^{\sigma_x}L$, where $x$ ranges over $G$. 
Choose
\begin{equation}\label{eqn:HK}
  H^{\bu}_K =  \coh^{ev}(L,k)\cdot \HH^0(K) ,
\end{equation}
where the superscript $ev$ indicates that we take only the evenly
graded part if the characteristic of $k$ is not 2 (to ensure
commutativity).
Then $H^0_K = \HH^0(K)$ by design.
The subalgebra $\coh^{ev}(L,k)$ is finitely generated
(see, e.g.,~\cite[Corollary~4.2.2]{Benson2}), 
and since $\HH^0(K)$ is finite dimensional, we see that $H^{\bu}_K$ is
finitely generated, and so (fg1) holds.

Next we explain that $\HH^{\bu}(K,M)$ is finitely generated over
$H^{\bu}_K$ for all finite dimensional $K$-bimodules $M$.
It suffices to show this for bimodules of the form $M= {}_xM_x
=  p_x\cdot M \cdot  p_x$, as these
are the bimodules that contribute to the Hochschild cohomology
$\HH^{\bu}(K,M)$.
By Lemma~\ref{lem:induced} and the Eckmann-Shapiro Lemma, 
for such a bimodule $M$, 
$$
   \HH^{\bu}(K, M) = \Ext^{\bu}_{(k^{\sigma_x} L)^e}(k^{\sigma_x}L, M)
  \simeq \Ext^{\bu}_L(k,M^{ad}).
$$
Now consider the action of $\coh^{\bu}(L,k)$ on $\HH^{\bu}(K,M)$ induced by
the inclusion of $\coh^{\bu}(L,k)$ into $\HH^{\bu}(K)$ 
given by Corollary~\ref{cor:embedding}.
We claim that this action  corresponds to that
of $\coh^{\bu}(L,k)$ on $\coh^{\bu}(L,M^{ad}) = \Ext^{\bu}_L(k,M^{ad})$
by cup product. Again this is essentially the same proof
as~\cite[Proposition~3.1]{SW}. 
Since $\coh^{\bu}(L,M^{ad})$ is finitely generated over $\coh^{\bu}(L,k)$
(see, e.g.,~\cite[Theorem~4.2.1]{Benson2}), we
conclude that $\coh^{\bu}(K,M)$ is finitely generated over $H^{\bu}_K$.
So (fg2) holds.
\end{proof}

Since $kL\natural_{\sigma}^{\tau} k^G$ satisfies (fg1) and (fg2)
by the theorem, support varieties
for $kL\natural_{\sigma}^{\tau} k^G$ may be defined as in Section~\ref{supp-var}.
We next examine these varieties
and some modules for a specific example.

\begin{example}\label{pos char}
Let $k$ be a field of characteristic $3$.
Consider the groups
$$G =\mathbb Z_2 = \langle h\rangle\ \ \ \mbox{ and } \ \ \
L = \mathbb Z_2\times \mathbb Z_2 \times \mathbb Z_3 \times \mathbb Z_3= \langle a\rangle \times \langle b\rangle \times
\langle c\rangle \times \langle d\rangle .$$
Let the generator $h$ of $G$ act on $L$
by interchanging the generators $c$ and $d$ (and leaving alone the generators $a$ and $b$).

A projective representation $\rho$ of the Klein four group
$\langle a\rangle \times \langle b\rangle$ may be given by
\[\rho (a) = \left(\begin{array}{cc} 0 & -1\\
1 & 0 \\
\end{array}\right), \quad
\rho (b) = \left(\begin{array}{cc} 1 & 0\\
0 & -1 \\
\end{array}\right),\quad
\rho (ab) = \left(\begin{array}{cc} 0 & 1\\
1 & 0 \\
\end{array}\right) . \]
There is an associated nontrivial $2$-cocycle $\alpha$, defined by
\[
  \alpha(l,m) = \rho(l)\rho(m)(\rho(lm))^{-1}
\]
for all $l,m\in \langle a \rangle \times \langle b \rangle$. 
We will also denote by $\alpha$ the trivial extension of this cocycle to $L$, that is
\[\alpha(a^{n_1}b^{m_1}c^{s_1}d^{t_1}, a^{n_2}b^{m_2}c^{s_2}d^{t_2}) = \alpha(a^{n_1}b^{m_1}, a^{n_2}b^{m_2})\]
for all integers $m_1,m_2,n_1,n_2, s_1, s_2, t_1, t_2$.
This gives rise to a $2$-cocycle $\sigma: L\times L\to (k^G)^{\times}$
in the following way. Let $$\sigma(l,m) =  p_1 + \alpha(l,m) p_{h}$$ for all
$l,m\in L$.
Thus $\sigma_h = \alpha$ and $\sigma_1$ is trivial.
We take $\tau: L\rightarrow k^G\ot k^G$ to be trivial,
that is, $\tau(l) = 1\ot 1$ for all $l\in L$.
Thus we leave $\tau$ out of the notation, writing $kL\natural _{\sigma} k^G$
for the resulting crossed coproduct.

One checks that the center of $K=kL\natural_{\sigma} k^G$  is spanned by $1\natural p_{h}$ together with
all $x\natural p_1$, $x\in L$.
So $\HH^0(K)\simeq Z(K)$ is spanned by these elements. 
Letting $H^{\bu}_K$ be as in~(\ref{eqn:HK}), 
$\V_K = (\cup_e \V_{kL}) \times G$ where $e$ ranges
over a set of primitive central idempotents in $kL$. 
It follows that 
$K$-modules of the form $M = M_{h}$ have support varieties of the form
$\V_{kL}(M_{h}) \times  h$, 
where $\V_{kL}(M_h)$ is the maximal ideal spectrum of the quotient
of $\coh^{ev}(L,k)$ by the annihilator of $\Ext^{\bu}_{k^{\sigma_h}L}(M_h,M_h)$
under the action given by $ - \ot M_h$
(which takes $kL$-modules to $k^{\sigma_h}L$-modules) followed by Yoneda composition. 
(The notation is not meant to indicate that this is a variety of a $kL$-module.)
Indecomposable $K$-modules of the form $M = M_1$ have support
varieties of the form $\V_{kL}(M_1)_e\times 1$, where the subscript $e$
is meant to indicate the block in which $M_1$ lies.
As $M_1$ is a $kL$-module, we may view $\V_{kL}(M_1)$ as the usual variety
of a $kL$-module.

Let $U_{\rho}$ be the
$k^{\alpha}(\langle a\rangle \times\langle b \rangle) $-module corresponding to
the projective representation $\rho$ described above. Let
$U = U_{\rho}\!\uparrow_{k^{\alpha}(\langle a\rangle \times\langle b \rangle )}
^{k^{\alpha}(\langle a\rangle \times\langle b \rangle\times\langle c\rangle )}$, and let $d$ act trivially on $U$ so that $U$ is a $k^{\alpha}L$-module.
Note that $U$ is not projective:
If it were, then restricting to $k\langle d \rangle$ would produce a
projective $k\langle d \rangle$-module since $k^{\alpha}L$ is
free over $k\langle d \rangle$. However the action of $d$ on $U$ is
trivial, and the characteristic of $k$ is the order of $d$, 
so the restriction of $U$ to $k\langle d\rangle$ is not projective. 
We will use $U$  to construct nonprojective $K$-modules with similar
properties to those in the last section and in~\cite{BW}.
For example, setting $M= U\ot kp_h$, a nonprojective $K$-module, we claim that
$M\ot M$ is projective: By Theorem~\ref{theor_modules_bicrossed char pos}, 
$M\ot M\simeq U \ot ({}^h\! U)\ot kp_1$. 
Restricting to $kL$, we have 
$U\ot ({}^h\! U)$, a
projective $kL$-module since its restriction to the Sylow 3-subgroup
$\langle c\rangle\times\langle d \rangle$ of $L$ is
isomorphic to a direct sum of copies of the free module
$k\langle c\rangle \ot k\langle d\rangle$.
It follows that $M\ot M$ is projective as a $K$-module.  
Note that here we did not use  varieties in our argument---the 
tensor product property is known for $kL$-modules, but
we are dealing with  the tensor product of a $k^{\sigma_h}L$-module and a
$k^{\sigma_h^{-1}}L$-module.
\end{example}

\section{Projective modules of quasitriangular Hopf algebras}\label{five}

In this section we show that the behavior of projective and nonprojective
modules occurring in the last two sections does not happen when the Hopf
algebra is quasitriangular.
More generally, we consider a finite dimensional Hopf algebra $A$ and a
finite dimensional $A$-module $M$ for which $M\ot M^*\simeq M^*\ot M$.

The category $A$-mod of finite dimensional (left) $A$-modules is a rigid monoidal category.
In particular for every finite dimensional $A$-module $M$, the composition
\begin{equation}
\label{eqn:rigidity}
   M\stackrel{\text{coev}_M\ot \text{id}_M}
     {\relbar\joinrel\relbar\joinrel\relbar\joinrel\relbar\joinrel\relbar\joinrel\longrightarrow}
   M\ot M^*\ot M
   \stackrel{\text{id}_M\ot \text{ev}_M}
     {\relbar\joinrel\relbar\joinrel\relbar\joinrel\relbar\joinrel\longrightarrow}
     M
\end{equation}
is the identity map on $M$, where ev$_M$ and coev$_M$ denote the canonical
evaluation and coevaluation maps.
See, for example, \cite[Section 2.1]{BK} for details.

The tensor product of a projective $A$-module with another module is projective
(see, for example, \cite[Proposition 3.1.5]{Benson}).
The tensor product of two nonprojective modules can be projective,
and when $A$ has the tensor product property, there are
necessary and sufficient 
conditions for this projectivity as we saw in Theorem~\ref{properties}(i):
The dimension of the variety of a module is 0 precisely when 
the module is projective.
Under a further assumption on the dual module,
 an elementary argument shows that tensor powers of
a nonprojective module are nonprojective:

\begin{theorem}\label{thm:rigid} 
Let $A$ be a finite dimensional  Hopf algebra, and let $n$ be a positive integer.
Let $M$ be a finite dimensional $A$-module for which $M\ot M^*\simeq M^*\ot M$.
Then $M$ is projective if and only if $M^{\ot n}$ is projective.
\end{theorem}

\begin{proof}
The tensor product of a projective module with any module is projective
(see, e.g.,~\cite[Proposition~3.1.5]{Benson}), so if
$M$ is a projective module, then $M^{\ot n}$ is projective.

For the converse, first consider the case $n=2$. Since the
composition of functions in (\ref{eqn:rigidity}) is the identity map,  
$M$ is a direct summand of $M\ot M^*\ot M$.
By hypothesis, $M\ot M^*\ot M\simeq M\ot M\ot M^*$.
If $M\ot M$ is projective, then  $M\ot M\ot M^*$ is projective, and so $M$
is a direct summand of a projective module and thus is projective itself.

The general statement of the converse now follows by induction on $n$.
Apply $M^{\ot (n-1)}\ot - $ to  (\ref{eqn:rigidity}) to see that $M^{\ot n}$
is a direct summand of $M^{\ot n}\ot M^*\ot M \simeq M^{\ot (n+1)}\ot M^*$.
Therefore if $M^{\ot (n+1)}$ is projective, then $M^{\ot n}$ is projective.
\end{proof}

\begin{remark}
Let $A$-stmod be the stable module category of $A$, that is,
objects are finite dimensional $A$-modules, and for any two objects $M, N$,
morphisms are
\[
    \underline{\text{Hom}}_A(M,N)  = \text{Hom}_A(M,N)/\text{PHom}_A(M,N)
\]
where $\text{PHom}_A(M,N)$ is the subspace consisting
of each $A$-module homomorphism
from $M$ to $N$ that factors through a projective module.
Thus each projective module is isomorphic to the 0 object in $A$-stmod.
If $A$ is quasitriangular,
we may view the above theorem as stating that there are no
nilpotent objects in the category $A$-stmod.
\end{remark}


\end{document}